\documentclass[a4, 12pt]{amsart}
\usepackage{amssymb}
\usepackage{amstext}
\usepackage{amsmath}
\usepackage{amscd}
\usepackage{latexsym}
\usepackage{amsfonts}
\usepackage{color}
\usepackage{enumerate}
\usepackage{textcomp}
\usepackage[all]{xy}
\usepackage{graphicx}
\usepackage{appendix}
\usepackage[pagebackref]{hyperref}
\usepackage{hyperref}
\hypersetup{
   colorlinks=true, 
    linkcolor=blue, 
    urlcolor=green, 
}

\DeclareFontFamily{OT1}{rsfs}{}
\DeclareFontShape{OT1}{rsfs}{n}{it}{<-> rsfs10}{}
\DeclareMathAlphabet{\mathscr}{OT1}{rsfs}{n}{it}

\usepackage[alphabetic,initials]{amsrefs}

\theoremstyle{plain}
\newtheorem{thm}{Theorem}[section]
\newtheorem*{thm*}{Theorem}
\newtheorem*{cor*}{Corollary}
\newtheorem*{defn*}{Definition}

\newtheorem{prop}[thm]{Proposition}
\newtheorem{lem}[thm]{Lemma}
\newtheorem*{lem*}{Lemma}
\newtheorem{cor}[thm]{Corollary}
\newtheorem{claim}[thm]{Claim}
\newtheorem*{claim*}{Claim}

\theoremstyle{definition}
\newtheorem{defn}[thm]{Definition}
\newtheorem{ex}[thm]{Example}
\newtheorem{rem}[thm]{Remark}

\newtheorem{nota}[thm]{Notation}

\newtheorem{setting}[thm]{Setting}

\theoremstyle{remark}
\newtheorem*{pf}{{\sl Proof}}

\numberwithin{equation}{thm}
\def\Hom{\mathrm{Hom}}

\def\Soc{\mathrm{Soc}}

\def\Ker{\mathrm{Ker}}
\def\Im{\mathrm{Im}}

\def\a{\mathfrak a}

\def\m{\mathfrak m}
\def\n{\mathfrak n}

\def\q{\mathfrak q}

\def\N{\Bbb N}

\def\Z{\Bbb Z}

\def\val{\mathrm {Val}}

\def\edim{\mathrm {edim}}
\def\reg{\mathrm {reg}}

\newcommand{\rmI}{\mathrm{I}}

\newcommand{\fkm}{\mathfrak{m}}
\newcommand{\fkn}{\mathfrak{n}}

\newcommand{\fkC}{\mathfrak{C}}

\def\depth{\mathrm{depth}}

\def\a_i{\underline {a_i}}

\def\Ass{\mathrm{Ass}}

\begin{document}

\title[]{Canonical  stretched rings}

\author[N.T.A. Hang]{Nguyen Thi Anh Hang}
\address{The Department of Mathematics, Thai Nguyen University of education.
	20 Luong Ngoc Quyen Street, Thai Nguyen City, Thai Nguyen Province, Viet Nam.}
\email{hangnthianh@gmail.com}

\author[D.V. Kien]{Do Van Kien}

\address{ Department of Mathematics, Hanoi Pedagogical University 2, Vinh Phuc, Viet Nam}
\email{dovankien@hpu2.edu.vn}

\author[H.L. Truong]{Hoang Le Truong}
\address{	Mathematik und Informatik, Universit\"{a}t des Saarlandes, Campus E2 4,
	D-66123 Saarbr\"{u}cken, Germany}
\address{Institute of Mathematics, VAST, 18 Hoang Quoc Viet Road, 10307
	Hanoi, Viet Nam}
\address{Thang Long Institute of Mathematics and Applied Sciences, Hanoi, Vietnam}
\email{hoang@math.uni-sb.de\\hltruong@math.ac.vn\\
	truonghoangle@gmail.com}
\dedicatory{Dedicated to Professor Nguyen Tu Cuong on the Occasion of His 70th Birthday}

\thanks{2020 {\em Mathematics Subject Classification\/}: 
13H10, 13H15.\\}
\keywords{ Cohen-Macaulay local ring, Canonical ideals, Stretched Cohen-Macaulay rings.}

\begin{abstract} 
In this paper, we introduce the concept of canonical stretched rings, sparse  stretched rings and maximum sparse ideals. Then we give characterizations  of canonical  stretched rings and sparse  stretched rings; and a characterization of Gorenstein rings in terms of their maximum sparse ideals.  Several explicit examples are provided along the paper to illustrate such rings.
\end{abstract}

\maketitle

\section{Introduction}

This paper studies a special class of Cohen-Macaulay local rings, which we call {\it canonical  stretched rings}. Originally, the definition of stretched Cohen-Macaulay rings was introduced by Sally in 1979 (cf. \cite{Sal79}). 
She developed in \cite{Sal79} a very nice theory of stretched Cohen-Macaulay rings and gave many interesting results such as a structure theorem for stretched Artinian local rings (\cite[Theorem 1.1]{Sal79} and cf. \cite[Theorem 3.1]{ElV08}). Recall that an Artinian local ring $(R,\m)$ is called a {\it stretched ring}, if $\m^2$ is principal.
After that, Sally  formulated the notion of ``stretched" for Cohen-Macaulay local rings of higher dimension in terms of a minimal reduction of $\m$.
In particular, a Cohen-Macaulay local ring $(S,\n)$ is said to be {\it stretched}, if there is a minimal reduction $\q$ of $\n$  such that  $\q$ is a parameter ideal of $S$ and $S/\q$ is a stretched Artinian local ring. In this paper, we give the new notion of ``stretched" for Cohen-Macaulay local rings in arbitrary dimension in terms of  their irreducible ideals. 
Notice that an ideal $I$ of $S$ is said to be {\it irreducible}, if $I$  is not written as the intersection of two larger ideals of $S$. We  call a  Cohen-Macaulay local ring $(S,\n)$ to be a {\it canonical stretched ring}, if there is an irreducible 
ideal $I\subseteq \n^2$ of $S$  such that $S/I$ is a stretched Artinian local ring. 

%
%
%

Now let us explain our notation, terminology and motivation about canonical stretched rings. Assume that $(S,\n)$ is a one-dimensional Cohen-Macaulay local ring possessing a canonical module $\omega_S$.  
We say that an ideal $I$ of $S$ is canonical, if $I\neq S$ and $I\cong \omega_S$ as $S$-modules.
It is well-known that because canonical ideals are faithful $S$-module \cite[Bemerkung 2.5]{HeK71}, they are $\n$-primary. Conversely, also by \cite[Satz 3.3]{HeK71}, an $\n$-primary ideal $I$ is a canonical ideal of $S$ if and only if $I$ is an  irreducible ideal of $S$. Then our definition of canonical  stretched rings is now stated as follows.
\begin{defn}
We say that $(S,\n)$ is a {\it canonical  stretched ring}, if $S$ is a one-dimensional Cohen-Macaulay local ring possessing a canonical ideal $I\subseteq \n^2$ such that $S/I$ is a stretched ring.
\end{defn}

Let us explain how this paper is organized. In Section 2, we revisit some well-known results about stretched rings and canonical ideals. In Section 3, we  shall give characterizations  of canonical  stretched rings. In Section 4, we shall introduce the notion of maximum sparse ideals in one-dimensional  Noetherian local domains and give a characterization of Gorenstein   rings in terms of their maximum sparse ideals. Moreover we will study a special class of canonical  stretched rings, which we call {\it sparse  stretched rings}. In the final section,  we will explore $3$-generated numerical semigroup rings over a field and their canonical  stretched property.

\section{Notation, conventions, and preliminary results}

In this section, we shall summarize preliminary results, which we need throughout this paper. Some of them are known but we will provide short proofs for the sake of completeness.


\subsection{}\label{setting1} “Let $(R, \fkm, k)$ be a local ring” identifies $\fkm$ as the unique maximal ideal of the commutative Noetherian local ring $R$ and $k$ as the residue field $k=R/\fkm$.

\begin{enumerate}
\item {\it The embedding dimension} of $R$ is defined to be $\edim(R) =\dim_k(\m/\m^2)$.
\item {\it The Hilbert function} $h_R$ of $R$ is defined by
$$h_R(i)=\ell_R(\fkm^i/\fkm^{i+1}) \quad \quad \text{for all } i\ge 0.$$
\item The function $h_R$ is called {\it non-decreasing} if $h_R(n - 1) \le h_R(n)$ for each $n \in \Bbb N$ and {\it decreasing} if there exists $\ell \in \Bbb N$ such that $h_R(\ell- 1) > h_R(\ell)$, in this case we say $h_R$ decreases at level $\ell$.

\item If $M$ is an $R$-module, then the {\it socle} of $M$ is the vector space $\Soc(M) =0 :_M\fkm$, $\mu_R(M)$ denotes the minimal number of generators of $M$ and we denote by $\ell_R(M)$ its length.
\item Let $I$ be an $\m$-primary ideal of $R$.  Then the {\it top socle degree} of $I$ is defined to be the maximum integer $s$ such that $(I:\fkm)\cap(\fkm^s+I)\neq (I:\fkm)\cap(\fkm^{s+1}+I)$ and denoted by $s(I)$. In particular, if $(R, \fkm, k)$ is an  Artinian local ring, then the top socle degree of $R$ is defined to be the maximum integer $s$ with $\fkm^s\neq0$, and denoted by $s(R)$. 
\item The {\it Hilbert-Samuel function} $H_R$ of $R$ is defined by $H_R(n) = \ell_R(R/\m^{n+1})$ for each $n \in \Bbb N$.
\item If $I$ is an ideal of $R$, then  $v_R(R/I)=v_R(I)=\max\{n\mid I\subseteq\m^n\}$.
\end{enumerate}

For each non-zero element $a\in R$, we define the {\it degree} of $a$ with respect to $\fkm$ by $\deg a=\max\{n\ge 0 \mid a\in \fkm^n\}$. Then the top socle degree of $I$ can be computed through the degree of an element in the socle of $R/I$, as follows.
\begin{prop}\label{degree} Let $(R,\m)$ be a local ring and $I$ an $\m$-primary ideal. Then  we have
	\begin{enumerate}[$\rm 1)$]
		\item $s(I)=\max\{\deg a\mid a\in (I:\fkm)\setminus I\}$, 
		\item $(I:\m)\cap \fkm^{s(R)}\setminus I=\{a\in (I:\fkm)\setminus I\mid \deg a=s(R)\}$.
	\end{enumerate}
\end{prop}
\begin{proof}
	1) Suppose that there exists $a\in (I:\fkm)\setminus I$ such that $\deg a>s(I)$. Then $$(I:\m)\cap \fkm^{\deg(a)}+I=(I:\m)\cap \fkm^{\deg(a)+i}+ I \text{ for all } i\ge 1.$$ Hence $(I:\m)\cap \fkm^{\deg(a)}+I=I$ which implies that $(I:\m)\cap \fkm^{\deg(a)}\subseteq I$. Hence, since $a\in (I:\fkm)\setminus I$, we get $a\notin \fkm^{\deg(a)}$ which is a contradiction. Therefore,  $\deg(a)\le s(I)$ for all $a\in (I:\fkm)\setminus I$. To show that there is an element $a\in (I:\fkm)\setminus I$ such that $\deg(a)=s(I)$, it is sufficient to prove 2).
	
	2) It is clear that the left hand side contains the right hand side. For reverse inclusion, take any $a\in (I:\m)\cap \fkm^{s(I)}\setminus I$. We then have $a\in (I:\fkm)\setminus I$. Moreover, by the proof of 1) we get $\deg(a)\le s(I)$. On the other hand, since $a\in\fkm^{s(I)}$, we have $\deg(a)\ge s(I)$. Therefore, $\deg(a)=s(I)$, which implies that the left hand side is contained in the right hand side, as required.
\end{proof}

\subsection{} The definition of stretched local rings was first introduced by Sally in 1979 (cf. \cite{Sal79}). 
\begin{defn}(cf. \cite{Sal79})\label{def} We say that a $d$-dimensional local ring $(S, \n)$ is {\it stretched}, if there is a minimal reduction $J$ of $\n$  such that $J$ is a parameter ideal of $S$ and $\ell_S(\frac{\n^2+J}{\n^3+J}) = 1$.
\end{defn}
If $R$ is an Artinian local ring, by the classical theorem of Macaulay on the shape of the Hilbert function of a standard graded algebra, the Hilbert function of a stretched Artinian local ring $R$ is given by:  
\begin{equation}\label{charbyHF}
\begin{tabular}{|c|c|c|c|c|c|}
\hline
$0$&$1$&$2$&$\cdots$&$s(R)$&$s(R)+1$\\
\hline
$1$&$\edim(R)$&$1$&$\cdots$&$1$&$0$\\
\hline
\end{tabular}
\end{equation}

\subsection{} Suppose that $(S,\n)$ is a Cohen-Macaulay local ring with $\dim S=1$.
Let $\omega_S$ denote the canonical module of $S$. Recall that for the $\n$-adic completion $\hat S$ of $S$, the canonical module $\omega_{\hat S}$ of $\hat S$ is defined by
$$\omega_{\hat S}:=\Hom_{\hat S}(H_\n^1(\hat S),\hat E),$$
where $H_\n^1(\hat S)$ denotes the first local cohomology module of $\hat S$ with respect to $\hat \n$ and $\hat E=E_{\hat S}(\hat S/\hat\n)$ the injective envelope of the $\hat S$-module $\hat S/\hat\n$. When $S$ is not necessarily $\n$-adically complete, the canonical module $\omega_S$ of $S$ is defined to be an $S$-module such that
$$\hat S\otimes_S \omega_{ S}\cong \omega_{\hat S}$$
as $ S$-modules.

The fundamental theory of canonical modules was developed in \cite{HeK71} by J. Herzog and E. Kunz. It is well-known that $S$ possesses the canonical module $\omega_S$ if and only if $S$ is a homomorphic image of a Gorenstein ring. In the present research,  we are interested  in rings which contain canonical ideals. Let us begin with the following.
\begin{defn} \label{canIdeal}An ideal $I$ of $S$ is said to be a {\it canonical ideal} of $S$, if $I\neq S$ and $I\cong \omega_S$ as $S$-modules.
\end{defn}

Here we remark that this definition implicitly assumes the existence of the canonical module $\omega_S$. Namely, the condition in Definition \ref{canIdeal} that $I\cong \omega_S$ as $S$-modules should be read to mean that $S$ possesses the canonical module $\omega_S$ and the ideal $I$ of $S$ is isomorphic to $\omega_S$ as $S$-modules. We then have the following result.
\begin{lem}\label{irr}
Let $(S, \n)$ be a local ring of dimension one, $Q$  its total ring of fractions and $\overline S\subseteq Q$ its normalization. 
Assume that $S$ possesses a canonical module. Then we have
\begin{enumerate}[$\rm 1)$]
\item $S$ has a canonical ideal if and only if the total ring of fractions of $\hat S$ is Gorenstein. {\rm(\cite{GMP13, HeK71})}.
\item  An $\n$-primary ideal $I$ is a canonical ideal if and only if $I$ is an irreducible ideal. {\rm(\cite[Satz 3.3]{HeK71}).}
\item There is a canonical module $\omega_S$ of $S$ such that $\omega_S:_Q\overline{S}=S:_Q \overline{S}$ and $S\subseteq \omega_S\subseteq\overline{S}$.{\rm (\cite[Lemma 3 ]{BrH92}).}
\end{enumerate}
\end{lem}

\section{Canonical  stretched rings}

\subsection{} Now we begin by setting the notation of this section.
\begin{setting}\label{nota}
 Let  $(S,\fkn,k)$  be a commutative Noetherian local  ring with the maximal ideal $\fkn$ and $\dim S=d$. Let $I$ be an $\fkn$-primary ideal of $S$ such that  $(R:=S/I,\fkm:=\fkn/I,k)$ is an Artinian local ring with  
the top socle degree $s$.
\end{setting}

The following lemma give the basic properties of the top socle degree and the invariant $v_S(R)$ (see \ref{setting1}).

\begin{lem}\label{Lem32}
Let $S$ and $R$ be as in Setting \ref{nota}. Then we have 
\begin{enumerate}[$\rm 1)$]
\item $s(R)=\max\{i\in\N\mid \fkn^i+I\neq \fkn^{i+1}+I\}$.
\item $v_S(R)=\inf\{i\mid h_R(i)< h_{S}(i)\}.$
\end{enumerate}
\end{lem}
\begin{proof}
$\rm 1)$ Let $t=\max\{i\in\N\mid \fkn^i+I\neq \fkn^{i+1}+I\}$. It is clear that $s(R)\le t$. Assume that $s(R)< t$. Then $(I:\n)\cap\n^t+I=(I:\n)\cap\n^{t+1}+I$. By Nakayama's lemma, we get $\fkn^t+I=\fkn^{t+1}+I$ which is impossible. We therefore have $t=s(R)$.\\
$\rm 2)$ It is now immediate from the definition of $v_S(R)$.
\end{proof}

Recall that in the case of Cohen-Macaulay algebras, Ooishi gave 
the definition of stretched Cohen-Macaulay algebras in terms of their Castelnuovo-Mumford regularity (\cite[Definition 14]{Ooi82}). In particular, if $A$ is a homogeneous Cohen-Macaulay algebra over a field $k$ then 
\begin{equation}\label{eq12}
\reg(A)\le e(A)+\dim(A)-\edim(A),
\end{equation}
where $\reg(A)$ is the Castelnuovo-Mumford regularity of $A$, $e(A)$ is the multiplicity of $A$, and $\edim(A)$ is the embedding dimension of $A$. Then $A$ is called a {\it stretched Cohen-Macaulay algebra}, if equality holds in \ref{eq12}. Therefore, if $(S, \n)$ is a local ring such that its associated graded ring $G(S) :=\bigoplus\limits_{n\ge 0} \n^n/\n^{n+1}$ is Cohen-Macaulay, then $G(S)$ is a stretched Cohen-Macaulay algebra exactly when $S$ is a stretched local ring in the sense of Sally. Notice that if $\dim A=0$ then the Castelnuovo-Mumford regularity of $A$ is the top socle degree of $A$.
In the present paper, the following theorem give a bound for the top socle degree and we use such bound  to define canonical stretched rings (Definition \ref{defCan}).

\begin{thm}\label{sdegI} We have 
	\begin{equation}\label{eq5}
		s(R)\le \ell_S(R)-\sum\limits_{i=0}^{v_S(R)-1}h_S(i)+v_S(R)-1,
	\end{equation}
	and if $I$ is not a power of  $\n$ then the following conditions are equivalent.
	\begin{enumerate}[$\rm 1)$]
		\item $s(R)= \ell_S(R)-\sum\limits_{i=0}^{v_S(R)-1}h_S(i)+v_S(R)-1$.
		\item $\mu(\m^{v_S(R)})= 1$.
		\item The Hilbert function of  $R$ is given by 
		\begin{equation}\label{eq7}
			\begin{tabular}{|c|c|c|c|c|c|c|c|c|c|}
				\hline
				$0$&$1$&$\cdots$&$v_S(R)-1$&$v_S(R)$&$\cdots$&$s(R)$&$s(R)+1$\\
				\hline
				$1$&$h_R(1)$&$\cdots$&$h_R(v_S(R)-1)$&$1$&$\cdots$&$1$&$0$\\
				\hline
			\end{tabular}.
		\end{equation}
	\end{enumerate}
\end{thm}
\begin{proof}
To prove the inequality, we put $t=\ell_S(R)-\sum\limits_{i=0}^{v_S(R)-1}h_S(i)$. Then
	$$\begin{aligned}
		t=\ell_S(S/I)-\sum\limits_{i=0}^{v_S(R)-1}h_S(i)&=\ell_S(S/I)-H_S(v_S(R)-1)\\
		&=\ell_S(S/I)-\ell_S(S/\n^{v_S(R)})=\ell_S(\n^{v_S(R)}/I).
	\end{aligned}$$
	Moreover, by the definition of $v_S(R)$, $I\subseteq\n^{v_S(R)}$. By taking a composition series 
	$$I=I_t\subsetneq I_{t-1}\subsetneq\ldots\subsetneq I_0=\n^{v_S(R)},$$ 
	we get $I_i/I_{i+1}\cong S/\n $ for all $0\le i\le t-1$. Therefore, $\n^{t+v_S(R)}\subseteq I$. Hence, we obtain 
		$s(R)\le t+v_S(R)-1$ as required.
		
We now assume that $I$ is not a power of  $\n$. The implication (3) to (2) is obvious. In the following, we give a proof of 1) to 3), 2) to 3), and 3) to 1).

	$1) \Rightarrow 3)$ Suppose the first condition holds, that is, $t=s(R)-v_s(R)+1$. Then, since $t=\ell_S(\n^{v_S(R)}/I)=\ell_R(\m^{v_S(R)})$ and $v_S(R)\le s(R)$ (by $I$ is not a power of  $\n$), we have
	$$\sum\limits_{i=v_S(R)}^{s(R)}\ell_R(\m^i/\m^{i+1})= s(R)-v_s(R)+1.$$
Hence, since $\ell_R(\m^i/\m^{i+1})\ge 1$ for all $i\le s(R)$, we get
	\begin{equation}\label{eq6}
		\ell_R(\m^i/\m^{i+1})=1
	\end{equation}
	for all $v_S(R)\le i\le s(R)$. 	Therefore, by definition of $v_S(R)$, the Hilbert function of $R$ is given by (\ref{eq7}).
	
	$2) \Rightarrow 3)$ Assume that $\mu(\m^{v_S(R)})= 1$. Then, by the classical theorem of Macaulay on the shape of the Hilbert function of a standard graded algebra,
	we also have $\ell_R(\m^i/\m^{i+1})=1$ for all $v_S(R)\le i\le s(R)$. We conclude that the Hilbert function of $R$ is given by (\ref{eq7}).
	
	$3) \Rightarrow 1)$ Suppose the Hilbert function of $R$ is given by (\ref{eq7}). Then we have
	$$\ell_S(R)=\sum_{i=0}^{v_S(R)-1}h_S(i)+s(R)-v_S(R)+1,$$
which implies that the equality holds in (\ref{eq5}).\\
\end{proof}


\begin{cor} \label{stret}Assume that $v_S(R)\ge 2$. Then  we have
	\begin{equation}\label{eq3}
		s(R)\le \ell_S(R)-\edim(R).
	\end{equation}
	If $I$ is not a power of $\n$ then the equality holds in (\ref{eq3}) if and only if $R$ is a stretched Artinian local ring. In this case,  we have
	\begin{equation}\label{eq4}
		v_S(R)=2 \text{ and } \mu(\m^2)=1.
	\end{equation}
\end{cor}
\begin{proof}
	Since $v_S(R)\ge 2$, $I\subseteq \n^2$. We have $\edim(R)=\dim_k\n/\n^2=h_S(1)$. Assume that there exists $0 \leqslant i \leqslant {v_S}(R) - 1$ such that $h_S(i)=0$. Then we have $\n^i=(0)$ which implies  $\n^{v_S(R)}=(0)$. Hence, by $I \subseteq\n^{v_S(R)}, I=(0)$ which is impossible. Therefore, ${h_S}(i) \geqslant 1$ for all $ 0 \leqslant i \leqslant v_S(R)-1$. It follows that
$$
		v_S(R)-1+\edim(R)\le \sum\limits_{i=0}^{v_S(R)-1}h_S(i).
$$
	Hence, by Theorem \ref{sdegI}, we get
\begin{equation}\label{eq8}
	s(R)\le  \ell_S(R)-\sum\limits_{i=0}^{v_S(R)-1}h_S(i)+v_S(R)-1 \le \ell_S(R) -\edim(R)
		\end{equation} which proves the inequality in (\ref{eq3}). 
			
Now we  assume that $R$ is a stretched Artinian local ring in the sense of Definition \ref{def}, then the Hilbert function of $R$ is given by (\ref{charbyHF}). Therefore, we have $$s(R) = \sum\limits_{i = 0}^{s(R)} h_R(i)-\edim (R) = \ell_S(R)-\edim (R),$$
as required. 
	
Conversely, we assume that the equality holds in (\ref{eq3}). Then, by (\ref{eq8}), we have  $$-\sum\limits_{i=0}^{v_S(R)-1}h_S(i)+v_S(R)-1 = -\edim(R).$$ 
Since $\edim(R)=h_S(1)$, we  have $h_S(0)+\sum\limits_{i=2}^{v_S(R)-1}h_S(i)+1=v_S(R)$.  Since $h_S(i)\ge 1$ for all $0\le i\le v_S(R)$, we must have $v_S(R)=2$.  By Theorem \ref{sdegI}, we have $\mu(\m^2)=1$. Hence, $R$ is a stretched Artinian local ring, as required.
\end{proof}
	
\begin{ex}
	Let $k[[t]]$ be the formal power series ring over a field $k$.
	\begin{enumerate}[$1)$]
		\item  Let $S = k[[t^3, t^7, t^8]]$, $I = (t^6, t^7)$, and $\n=(t^3,t^7,t^8)$. Then $I$ is an $\n$-primary ideal of $S$. Put $R=S/I$. Observe that $S/I$ has the length 4 and embedding dimension $\edim(S/I)=2$. We have $\fkn^3\subseteq I$. It implies that $0<s(R)\le 2$. Moreover, because the complement $((I:\fkn)\cap \fkn^2+I)\setminus ((I:\fkn)\cap \fkn^3+I)$ is not empty (it contains $t^{11}$), we have $s(R)\ge 2$. Hence $s(R)=2=\ell_S(S/I)-\edim(S/I).$ Therefore, $S$ is a stretched ring.
		\item Let $e \ge 3$ be an arbitrary integer. We consider the local ring
		$$S = k[[t^e,t^{e+1},\ldots,t^{2e-1}]]$$
		in the formal power series ring $k[[t]]$. Let $I =
		(t^e,t^{e+1},\ldots,t^{2e-2})$ is an $\n$-primary ideal of $S$. Then $S/I$ has the length 2 and embedding dimension $\edim(S/I)=1$.  We have $\fkn^2\subseteq I\subsetneq \fkn$, whence $0<s(S/I)\le 1$. We get $s(S/I)=1=\ell_S(S/I)-\edim(S/I)$.  Hence, $S$ is also a stretched ring.
	\end{enumerate}
\end{ex}

Now we define canonical  stretched rings and give their characterizations.
\begin{defn}\label{defCan}
We say that $(S,\n)$ is a {\it canonical  stretched local ring}, if there is an irreducible ideal $I\subseteq \n^2$ of $S$  such that $R=S/I$ is a stretched Artinian local ring. 
\end{defn}

We now develop the theory of canonical  stretched rings. 
 Let us begin with the following.

\begin{lem}\label{basis}
Assume that $S$ is a canonical  stretched ring. Then, 
there is  an irreducible ideal $I\subseteq \n^2$ of $S$ and a basis $z_1,\ldots, z_{\edim(S)}$ for $\n$  such that
\begin{enumerate}[$\rm i)$]
\item $\n^n+I=(z_1^n)+I$ for all $n\ge 2$ and $z_1z_i\in I$ for all $2\le i\le \edim(S)$.
\item For all $2\le i,\ j\le\edim(S)$, either $z_iz_j\in I$ or $z_iz_j-u_{ij}z_1^{s(S/I)}\in I$ for some unit $u_{ij}$ of $S$.
\item For each $2\le i\le \edim(S)$ there exists $j\ge 2$ such that $z_iz_j-u_{ij}z_1^{s(S/I)}\in I$ for some unit $u_{ij}$ of $S$. 
\end{enumerate}
\end{lem}
\begin{proof}
 Since $S$ is a canonical  stretched ring, there exists an irreducible ideal $I\subseteq \n^2$ of $S$ such that $S/I$ is a stretched ring. Let $R=S/I$, $\m=\n/I$, and $s=s(R)$. Since $I$ is  an irreducible ideal, $R=S/I$ is a Gorenstein ring. In other words,  $R=S/I$ is a zero dimensional stretched local Gorenstein ring. Then, by \cite[Theorem 1.1]{Sal79}, there are elements $z_1,z_2,...,z_{\edim S}$ in $\n$ 
satisfying the following conditions.
\begin{enumerate}[$\rm a)$] 
\item $\bar z_1,\ldots, \bar z_{\edim(S)}$  is a basis for $\m$, where $\bar z_i$ denote the image of $z_i$ in $R$ for all $i$.
\item $ (\bar z_2,\bar z_3,...,\bar z_{\edim S})\subseteq 0:_R\m^2$.
\item $\m^n=(\bar z_1^n) \text{ for all } n=2,3,...,\ell_S(R)-\edim S$.
\item $\bar z_1\in (0:_R(\bar z_2,...,\bar z_{\edim S}))$.
\item For each $2\le i\le\edim S$ there is $2\le j\le \edim S$ such that $\bar z_i\bar z_j=u_{ij}\bar z_1^{s(R)}$ for some units $u_{ij}$ in $R$.
\end{enumerate}
  It follows from $R$ is a stretched ring and Corollary \ref{stret} that $s=\ell_S(R)-\edim S$. We therefore get the assertions i) and iii). 

Now we show  the assertion ii). Indeed, assume that there exists $2\le i,\ j\le\edim(S)$ such that $\bar z_i\bar z_j\neq 0$. Then by the assertion c) we have $\bar z_i\bar z_j=\bar u_{ij}\bar z_1^n$, for some $n\ge2$ and some unit $\bar u_{ij}\in R$. On the other hand, by the assertion d), we have $\bar u_{ij}\bar z_1^{n+1}=\bar z_1\bar z_i\bar z_j=0$. By the definition of $s$, we have $n=s$. Hence $\bar z_i\bar z_j=0$ or $\bar z_i\bar z_j=\bar u_{ij}\bar z_1^s$ for some unit $u_{ij}$ of $S$. Therefore we get  the assertions ii).

\end{proof}
\begin{thm}
Assume that the Hilbert function of $S$ is non-decreasing. Then the following statements are equivalent.
\begin{enumerate}[$\rm 1)$]
\item $S$ is a canonical  stretched ring.
\item There exists an $\n$-primary ideal $I\subseteq \n^2$ such that   $\mu(\n^{v_S(I)}/I)=1$ and $I:\n\subseteq \n^{v_S(I)}$. 
\end{enumerate}
When this is the case, $v_S(I)=2.$  
\end{thm}
\begin{proof}
$1) \Rightarrow 2)$   Since $S$ is a canonical  stretched ring, there exists  an irreducible ideal $I\subseteq \n^2$ of $S$ such that $S/I$ is a stretched ring. 
By Corollary \ref{stret}, we have $v_S(I)=2, \mu(\n^{v_S(I)}/I)=1$ and $I:\n\subseteq \n^{v_S(I)}$.\\
$2) \Rightarrow 1)$ 
Let $R=S/I$ and $\m=\n/I$. It follows from $\mu(\m^{v_S(I)})=1$ and Corollary \ref{stret} that the Hilbert function of  $R$ is given by \ref{eq7}.  
Therefore, by Theorem 3.4 in \cite{Sha14} , we have
$$h_S(v_S(R)-1)-h_S(1)+1\le \dim_k(0:_R\fkm).$$
Moreover, since $0:_R\fkm \subseteq \m^{v_S(R)}$ and  $\mu(\m^{v_S(I)})=1$, we have $\dim_k(0:_R\m)=1$. Then   $h_S(v_S(R)-1)-h_S(1)\le 0$. Since the Hilbert function $h_{S}$ of $S$ is non-decreasing, we have $v_S(R)=2$. Then by Corollary \ref{stret},  $S/I$ is a stretched ring. 

On the other hand, since $\dim_k(0:_R\m)=1$, $I$ is an irreducible ideal.   Hence, $S$ is a canonical  stretched ring.
\end{proof}





\begin{ex}
Let $k[[X,Y,Z]]$ be the formal power series ring over a field $k$ and   $$S=k[[X,Y,Z]]/(X,Y)\cap(Y,Z)\cap(Z,X)$$
the local ring with the maximal ideal $\n=(x,y,z)$, where $x$, $y$, and $z$ denote the images of $X$, $Y$, and $Z$ in $S$, respectively.  Then $S$ is a reduced ring but not an integral domain with Cohen-Macaulay type $r(S)=2$ and $\dim S=1$.

Let $s\ge 2$, $I=(xy,yz,zx,x^s+y^2,x^s+z^2)$ and  $R=S/I$, $\m=\n/I$.  Then $I\subseteq \n^2$, $I:\n=(xy,yz,zx,x^s,y^2,z^2)$ and so $\ell_S((I:\n)/I)=1$. Hence, for all $s\ge2$, $I$ is a canonical ideal of $S$. Moreover, the Hilbert functions of  $S$ and $R$ are

\begin{equation}\label{eq11}
\begin{tabular}{|c|c|c|c|c|c|c|c|c|c|}
\hline
$0$&$1$&$2$&$3$&$\ldots$\\
\hline
$1$&$3$&$3$&$3$&$\ldots$\\
\hline
\end{tabular}
\quad\quad \text{and} \quad\quad
\begin{tabular}{|c|c|c|c|c|c|c|c|c|c|}
\hline
$0$&$1$&$2$&$\ldots$&$s$&$s+1$\\
\hline
$1$&$3$&$1$&$\ldots$&$1$&$0$\\
\hline
\end{tabular}
\end{equation}
Hence $S$ is a canonical  stretched ring. 

\end{ex}

\begin{ex}\label{exam1} Let $e$ be a positive integer, $e\ge 3$. We consider the local ring
$$S=k[[t^e,t^{e+1},\ldots,t^{2e-1}]]$$
in the formal power series ring $k[[t]]$ over a field $k$. We put $I=(t^{2e},t^{2e+1}\ldots,t^{3e-2})$. Then $I$ is a canonical ideal of $S$ and Hilbert function of $S/I$ is 
$$
h_{S/I}(i)=\begin{cases}
1& \text{ if } i=0,\\
e& \text{ if } i=1,\\
1& \text{ if } i=2,\\
0& \text{ if } i\ge3.
\end{cases}$$
Hence $S$ is a canonical  stretched ring.

\end{ex}


\section{Maximum sparse ideals}
In this section, we  study a special class of canonical  stretched rings, which we call {\it sparse  stretched rings}. Let us begin with  notations.

A subset $H$ of the set of non-negative integers $\N$ is called a {\it numerical semigroup}, if it contains $0$, is closed under addition and has a finite complement in $\N$. The largest integer not belong to $H$ is called the {\it Frobenius number} of $H$ and denoted by $g(H)$. We denote $\delta(H):=\sharp(\N\setminus H)$ and call it the {\it genus} of the numerical semigroup $H$. Then we obtain an upper bound of the Frobenius number of $H$ by its genus. Indeed, we observe that for all $s\in H$ then $g(H)-s\notin H$. Therefore, $\delta(H)\ge \sharp\{s\in H\mid s<g(H)\}$. But $\sharp\{s\in H\mid s<g(H)\}+\delta(H)=g(H)+1$, we get $g(H)\le 2\delta(H)-1$. This result can be extended to the Frobenius number of an ideal in $H$.

We recall that a subset $I$ of a numerical semigroup $H$ is said to be an {\it ideal} of $H$, if $I + H\subseteq I$.  A {\it relative ideal} of $H$ is a subset $F$ of $\Z$ with the property that  $F+H\subseteq F$ and $F +h\subseteq H$ for some $h\in H$. An relative ideal $\Omega$ is called a {\it canonical ideal} of the semigroup $H$ if $\Omega-(\Omega-F)=F$ for every relative ideal $F$ of $H$. Because the complement of $H$ in $\N$ is finite, it is not difficult to see that the complement of an ideal $I$ of $H$ in $\N$ is finite as well. We also call the largest integer which not belong to the ideal $I$ to be the Frobenius number of $I$ and denote by $g(I)$. We denote the difference of the ideal $I$ with respect to $H$ by $d(I)$, it is exactly the cardinality of $H\setminus I$.  In \cite{Bra19}, M. Bras-Amor\'{o}s gave an upper bound on the Frobenius number of an ideal which extends the upper bound for the Frobenius number of a numerical semigroup.

\vspace*{0.5cm}
\noindent{\bf Theorem} (\cite[Theorem 1]{Bra19}).\label{semigroup} Let $I$ be an ideal of a numerical semigroup $H$. Then
	$$g(I)\le d(I)+2\delta(H)-1.$$

The ideals for which the Frobenius number attains the bound are called {\it maximum sparse ideals} of $H$. In \cite[Theorem 2]{Bra19} they gave a characterization to a proper ideal that is maximum sparse. It is known that \cite[Theorem 2]{BLV14} the class of maximum sparse ideals of $H$ contains the class of canonical ideals. Moreover, we can use maximum sparse ideals to characterize  the symmetry of a numerical semigroup. More precisely, a numerical semigroup $H$ is symmetric if and only if there is a  maximum sparse principal ideal \cite[Corollary 1]{BLV14}.  In this section, we generalize the notion of maximum sparse ideals to a one-dimensional Noetherian local domain.

\begin{setting}\label{set1} Throughout this section, let $(S,\fkn)$ be a one-dimensional Noetherian local domain with the infinite residue field $k$ and the  quotient field $K$. We assume that $S$ is not regular with normalization $\overline{S} \subseteq K$. We suppose that $\overline{S}$ is a DVR and a finite $S$-module, i.e., $S$ is analytically irreducible. Let $t \in \overline{S}$ be a uniformizing parameter for $\overline{S}$, so that $t\overline{S}$ is the maximal ideal of $\overline{S}$. We also suppose that the field $k$ is isomorphic to the residue field $\overline{S}/t\overline{S}$, i.e., $S$ is residually rational. We denote the usual valuation on $K$ associated to $S$ by $\val$.
	
In this setting for a pair of non-zero fractional ideals $J\subseteq I$, it is possible to compute (cf. \cite{Mat71}) the length of the $S$-module $I/J$ by means of valuations, that is,
$$\ell_S(I/J)=|\val(I)\setminus\val(J)|.$$

We denote $\val(S) := \{\val(a) \mid a \in S, a\neq 0\} \subseteq \N$ to be the value semigroup of $S$. Since the conductor $\fkC := (S :_K \overline{S})$ is an ideal of both $S$ and $\overline{S}$, there exists a positive integer $c$ so that $\fkC = t^c\overline{S}$, $\ell_S(\overline{S}/\fkC) = c$ and $c \in \val(S)$. Furthermore, we denote $\delta := \ell_S(\overline{S}/S)$ the number of gaps of the semigroup $\val(S)$. This means that $\delta$ is the genus of $\val(S)$. Let $r := \ell_S((S : \fkn)_K/S)$ be the Cohen Macaulay type of $S$. 
We list the elements of $\val(S)$ in order of size: $\val(S) := \{s_i\}_{i\ge0}$, where $s_0 = 0$ and $s_i <s_{i+1}$, for every $i\ge0$. We put $e:=s_1$ the multiplicity of $S$ and $n=c-\delta$ the number such that $s_n = c$.   For every $i \ge 0$, let $S_i$ denote the ideal of elements
whose values are bounded by $s_i$, that is,
$$S_i :=\{a\in S\mid \val(a)\ge s_i\}.$$
\end{setting}
\begin{nota} We assume as in the Setting \ref{set1}. The following is a list of symbols and relations to be used in the sequel. For fractional ideals $I$, $J$:
\begin{enumerate}[$\rm 1)$]
\item $(I:J):=(I:_KJ)=\{a\in K\mid aJ\subseteq I\}.$
\item $\fkC_I=(I:\overline{S})$, the largest $\overline S$-ideal contained in $I$.
\end{enumerate}
Let $I$ be a proper ideal of $S$. We denote
\begin{enumerate}[$\rm 1)$]
\item $c(I):= \ell_S(S/\fkC_I)$, so that $t^{c(I)}\overline{S}=\fkC_I$; $c\le c(I)$ since $\fkC_I \subseteq \fkC$.
\item $n_I$ is a number such that $s_{n_I} =c(I)$, $\fkC_I =S_{n_I}$, $n_I =\ell_S(S/\fkC_I)=c(I)-\delta$.
\end{enumerate}

\end{nota}

The ideals $S_i$ give a strictly decreasing sequence
$$S=S_0 \varsupsetneq S_1 =\fkn \varsupsetneq S_2 \varsupsetneq \ldots \varsupsetneq S_n =\fkC \varsupsetneq S_{n+1} \varsupsetneq\ldots,$$
which induces the chain of duals:
$$S \varsubsetneq (S :S_1) \varsubsetneq \ldots \varsubsetneq(S :S_n)= \overline{S} \varsubsetneq (S :S_{n+1})=t^{-1}S\varsubsetneq\ldots$$

For each non-negative integer $i$, we define $D(i) = \{s_j\mid s_j \le s_i \text{ and } s_i - s_j \in \val(S)\}$. The set $D(i)$ is often called the set of {\it divisors} of $s_i$, and its cardinality is denoted by $\nu_i = |D(i)|$. 

Let $I$ be an ideal of $S$. We define the {\it Frobenius number} of the ideal $I$ to be the largest valuation of an element not belongs to $I$ denoted by $g(I)$, that is, 
$$g(I)=\max\{\val(a)\mid  a\not\in I \text{ and } a\in K  \}.
$$
 With the above notations, we have the following results.
\begin{lem}\label{properties1}
 Let $I$ be an $\fkn$-primary ideal of $S$. Then we have
 $$\begin{aligned}
 g(I)&=c(I)-1\\
&=\max\{\val(a)\mid  a\not\in I\text{ and } a\in I:\fkn \}.
\end{aligned}$$
\end{lem}
\begin{proof}
 First, we show that $g(I)=c(I)-1$. Indeed, for all $a\in K$ such that $\val(a)>c(I)-1$ then $\val(a)\ge s_{n_I}$. It follows that $a\in S_{n_I}=I:\overline{S}$, whence $a\in I.$ On the other hand, let $b\in K$ such that $\val(b)=c(I)-1$. Then $\val(b)<c(I)=s_{n_I}$. It implies that $b\notin S_{n_I}$. Hence, $b\notin I$. Therefore,  we have $c(I)-1=g(I),$
 as required.	
 
 Now let $g^\prime=\max\{\val(a)\mid  a\not\in I\text{ and } a\in I:\fkn \}$. Then it is clear that $g^\prime(I)\le g(I)$. On the other hand, assume that there is an element $a\in K\setminus I$ such that $\val(a)=g(I)$ but $a\notin I:\fkn$. Then there exists $\alpha\in\fkn$ such that $a\alpha\notin I$. It implies that
$\val(a\alpha)\le \val(a)$. Hence, $\val(a)+\val(\alpha)\le\val(a)$. We get $\val(\alpha)\le 0$ which is a contradiction to $\alpha\in \fkn$. Thus, $g(I)\le g^\prime(I)$, as required.
\end{proof}
\begin{lem}\label{properties}  Let $I$ be an $\fkn$-primary ideal of $S$. Then the following assertions are true.	
\begin{enumerate}[$\rm 1)$]
\item $((I:\fkn)\cap \fkn^{s(S/I)})\setminus I)=\{a\not\in I\mid \val(a)= g(I)\}$.
\item  If $I=I_1\cap I_2$ then $g(I)=\max\{g(I_1),g(I_2)\}$. Moreover, if $I=\bigcap_{i=1}^{s}I_i$, where $I_i$ are irreducible then $g(I)=\max_{i=1}^s\{g(I_i)\}$.
\item We have
$$\nu_{n_I}\le \ell_S(S/I)+1.$$ 
\end{enumerate}
\end{lem}
\begin{proof}
$\rm 1)$  Recall that the degree of non-zero element $a\in S$ is defined by $\deg a=\max\{n\ge 0 \mid a\in \fkn^n\}$.  Now, we show the following claim.
\begin{claim}
	For all $a,b\in (I:\fkn)\setminus I$, if $\deg b\le \deg a$ then $a\overline{S}\subseteq b\overline{S}$.
\end{claim} 
\begin{proof}
	Let $a,b\in (I:\fkn)\setminus I$ such that $\deg b\le \deg a$. Suppose  $\deg b< \deg a$ but $a\overline{S}\nsubseteq b\overline{S}$. Then since $\overline{S}$ is a DVR, we have $b\overline{S}\subseteq a\overline{S}$. It follows that $b=as$ for some $s\in\overline{S}$. Since $\deg b< \deg a$, we get $b\notin \fkn^{\deg(a)}$ which implies that $as\notin \fkn^{\deg(a)}$. Hence, since $a\in\fkn^{\deg(a)}$, one has
	$$\val(a)\ge\deg(a)>\val(as)=\val(a)+\val(s).$$
	Therefore, $\val(s)<0$. This is impossible because  $s\in\overline{S}$. Thus, if $\deg b< \deg a$ then $a\overline{S}\subseteq b\overline{S}$.
	Now if $\deg(a)=\deg(b)$ then $a\overline{S}= b\overline{S}$. Indeed, without loss of generality we may assume $b\overline{S}\subseteq a\overline{S}$. Then $b=au$ for some $u\in \overline{S}$. Since $b\in \fkn^{\deg(a)}\setminus\fkn^{\deg(a)+1}$, we have
	$$\deg(a)+1>\val(a)+\val(u)\ge\deg(a).$$ Therefore, $\val(a)+\val(u)=\deg(a)$. We get $\val(u)=0$. This implies that $u$ is invertible and hence $a\overline{S}= b\overline{S}$.
\end{proof}
 We now will show that $$(I:\fkn)\cap\fkn^{s(S/I)}\setminus I\subseteq \{a\in K\setminus I \mid \val(a)=g(I)\}.$$
Indeed, for all $a\in (I:\fkn)\cap\fkn^{s(S/I)}\setminus I$ and for all $b\in (I:\fkn)\setminus I$, Proposition \ref{degree} implies that $\deg(b)\le \deg(a)$. Hence, $a\overline{S}\subseteq b\overline{S}$ so that $\val(b)\le \val(a)$. Therefore, by Lemma \ref{properties1} we get $\val(a)=g(I)$ as desired.  For the reverse inclusion, we take $a\in K\setminus I$ such that $\val(a)=g(I).$ Then we have $a\fkn\subseteq I$. Because if otherwise then there exists $b\in \fkn$ such that $ab\notin I$. It implies that
$$\val(a)+\val(b)=\val(ab)\le \val(a),$$
whence $\val(b)=0$. This is impossible because $b\in \fkn$. Thus, $a\in I:\fkn$. On the other hand, we take an element $b\in  ((I:\fkn)\cap\fkn^{s(S/I)})\setminus I$, then $\val(b)=g(I)$. Therefore $\val(a)=\val(b)$ and we get $a=bu$ for some a unit $u$ in $S$. Hence, $a\in ((I:\fkn)\cap \fkn^{s(S/I)})\setminus I$ as desired, that is, $(I:\fkn)\cap\fkn^{s(S/I)}\setminus I\supseteq \{a\in K\setminus I \mid \val(a)=g(I)\}.$

$\rm 2)$ Because $K\setminus I\supseteq(K\setminus I_1)\cup(K\setminus I_2)$, we get $g(I)\ge g(I_i)$ for all $i=1,2.$ We suppose that $g(I)> g(I_i)$ for all $i=1,2$. Then there exists $a_i\in I_i$ such that $\val(a_i)=g(I)$ for all $i=1,2$. Hence since $a_2\notin I_1$ we get $g(I)=\val(a_2)\le g(I_1)$ which is a contradiction. Thus $g(I)=\max\{g(I_1),g(I_2)\}$. Now we decompose each $I_1, I_2$ into the intersection of two ideals and appy the result above for $I_1, I_2$. Continuing the process, we get the latter assertion in 2). Here note that the process will stop when $I_i$'s are irreducible.

$\rm 3)$ One has $$\begin{aligned}\nu_{n_I}&=|\{s_j\mid s_j \le s_{n_I} \text{ and }  s_{n_I-1} -s_j \in \val(R)\}|\\
&\le |\{s_j\mid s_j \le s_{n_I} \text{ and } s_j-e \not\in I\}|+1\\
&=|\{s_j\mid s_j \le s_{n_I} \text{ and } s_j-e \not\in I\}\cup\{s_{n_I}-1\}|\\
&\le \ell_S(S/I)+1.
\end{aligned}$$
\end{proof}
Combining Proposition \ref{degree} and Lemma \ref{properties} we get the following result.
\begin{cor}
$\{a\in K\setminus I \mid \val(a)=g(I)\}=\{a\in K\setminus I \mid \deg(a)=s(S/I)\}.$
\end{cor}
\begin{thm}\label{bounds}
Let $I$ be an $\fkn$-primary ideal of $S$. Then we have 
\begin{equation}\label{eq10}
s(S/I)\le s(S/I)e(\fkn)\le g(I).
\end{equation}
Equality holds in \ref{eq10} if and only if  $S$ is a DVR.
\end{thm}
\begin{proof}
Let $\Ass(S/I)=\{Q_i\}_{1\le i\le n}$ be the set of associated prime ideals of $S/I$. Since $\fkn^{s(S/I)}\nsubseteq I$,  $(\fkn^{s(S/I)}\cap I+\fkn^{s(S/I)+1})/\fkn^{s(S/I)+1}$ and $(\fkn^{s(S/I)}\cap Q_i+\fkn^{s(S/I)+1})/\fkn^{s(S/I)+1}$ are proper vector subspaces of  $\fkn^{s(S/I)}/\fkn^{s(S/I)+1}$.  Moreover, since $|S/\fkn|=\infty$, we get $$\frac{\fkn^{s(S/I)}\cap I+\fkn^{s(S/I)+1}}{\fkn^{s(S/I)+1}}\cup(\bigcup_{i=1}^n \frac{\fkn^{s(S/I)}\cap Q_i+\fkn^{s(S/I)+1}}{\fkn^{s(S/I)+1}})\neq \frac{\fkn^{s(S/I)}}{\fkn^{s(S/I)+1}}.$$ 
Hence, we can choose $a\in\fkn^{s(S/I)}\setminus (I\cup\bigcup_{i=1}^n Q_i)$ which is a superficial element of $\fkn^{s(S/I)}$. Then since the analytic spread of $\fkn^{s(S/I)}$ is $1$,  $aR$ is a reduction of $\fkn^{s(S/I)}$. Therefore, $e(aS)=e(\fkn^{s(S/I)})\ge s(S/I)e(\fkn)$. On the other hand, we have
$$e(aS)=\ell_S(S/aS)=\ell_S(S/aS)+\ell_S(\overline{S}/S)-\ell_S(a\overline{S}/aS)=\ell_S(\overline{S}/a\overline{S})=\val(a)\le g(I).$$ We therefore get the second inequality in the assertion. 
\end{proof}
The following lemma give an explicit formula for the number of the set of divisors of $s_i$. 
\begin{lem}{\cite[Lemma 51]{Bra13}}\label{div} For each $i\in \Bbb N,$ 
 let $\delta(i)$ be the number of gaps in the interval from $1$ to $s_i - 1$ and let $G(i)$ be the number of pairs of gaps whose sum equals $s_i$. Then, $\nu_i = i - \delta(i) + G(i) + 1$.
\end{lem}

We have the following theorem which is an extension of a result of M. Bras-Amor\'{o}s \cite[Theorem 1]{Bra19}.
\begin{thm}\label{main}
Let $I$ be an $\fkn$-primary ideal of $S$. Then
$$g(I)+1\le\ell_S(S/I)+2\delta.$$
\end{thm}
\begin{proof}
It is straightforward to see that the intersection of two ideals satisfying the result also satisfies the result. 
Now, by Lemma \ref{properties} 2), it will be enough to show that the result holds for the irreducible ideal $I$. 
By Lemma \ref{properties} 3), it will be enough to show that $\nu_{n_I} -1 + 2\delta \ge c(I)$. Indeed, since  $s_{n_I}\ge c$, one has $\delta({n_I})=\delta$, $s_{n_I} ={n_I}+\delta$, and as a consequence of Lemma \ref{div}, $$\nu_{n_I} -1 + 2\delta= ({n_I} - \delta + G({n_I}) + 1) -1+ 2\delta = {n_I} + \delta + G({n_I}) = s_{n_I} + G({n_I}) \ge s_{n_I}. $$

\end{proof}
\begin{cor} Let $I$ be an $\fkn$-primary ideal of $S$. Then
$$\ell_S(I/\fkC_I)\le\delta.$$
\end{cor}
\begin{proof}
Thanks to Theorem \ref{main} we have
\begin{align*}
&\ell_S(\overline{S}/\fkC_I)=\delta +\ell_S(R/\fkC_I)=c(I)\le\ell_S(S/I)+2\delta.
\end{align*}
Therefore, $ \ell_S(R/\fkC_I)\le\ell_S(S/I)+\delta$. Hence, $\ell_S(I/\fkC_I)\le\delta.$

\end{proof}

\begin{defn}
 Let $(S,\fkn)$ denote a Cohen-Macaulay local ring with $\dim S=1$. Let $K$ be the total quotient ring of $S$ and $\overline{S}$ the integral closure of $S$ in $K$. An $\fkn$-primary ideal $I$  of $S$ is called {\it maximum sparse}, if $\ell_S(I/I:_K\overline{S})=\ell_S(\overline{S}/S)$.
\end{defn}
\begin{ex}
Let $S=k[[H]]=k[[t^s\mid s\in H]]\subseteq k[[t]]$ be a semigroup ring of a numerical semigroup $H$ and $I$ be an ideal of $S$. Then $\val(S)=H$ and $\val(I)$ is an ideal of $H$. It is easy to see that if $I$ is maximum sparse then $\val(I)$ is also a maximum sparse ideal. The converse also holds.
\end{ex}
\begin{ex}\label{maxsparse}
Let $S=k[[t^4,t^6+t^7,t^{15}]]\subseteq\overline{S}=k[[t]]$, where char$(k)\neq 2$. It is not difficult to see that $\val(S)=\left<4,6,13,15\right>$. Let $I=(t^{12}-t^{16},t^{14}+t^{15}+t^{16}+t^{17})S$. By a direct computation, we get
$\ell_S(\overline{S}/S)=7, g(I)=23 $ and $ \ell_S(S/I)=10$. Therefore, 
$$g(I)+1=2\ell_S(\bar{S}/S)+\ell_S(S/I).$$
Hence, $I$ is a maximal sparse ideal. Notice that $S$ is not a numerical semigroup ring. 
\end{ex}
\begin{rem}\label{exists1} Let $S$ be as in Setting \ref{set1}. Then $S$ always contains a maximal sparse ideal. Indeed, if we let $\val(S)=\{\lambda_0<\lambda_1<\cdots\}$ then since $\sharp(\N\setminus \val(S))<\infty$, there exists $i>0$ such that $\lambda_i$ less than twice of the Frobenius number of $\val(S)$. Then $G(i)=0$ so that $I:=\val(S)\setminus D(i)$ is a maximum sparse ideal of $\val(S)$. Hence, the preimage $\val^{-1}(I)$ is a maximum sparse ideal in $S$.
\end{rem}
Notice that the canonical module of $S$ in Example \ref{maxsparse} is $\omega_S=R+\frac{t^2}{1-t}R$ and $I=(t^{12}-t^{16})\omega_S$ so that $I\cong \omega_S$. Hence, $I$ is also a canonical ideal of $S$. In general, this is also true for every maximum ideal.

\begin{lem}\label{ex1}
If $I$ is a maximum sparse ideal of $S$ then $I$ is a canonical ideal.
\end{lem}
\begin{proof}
 Suppose that $I$ is a maximum sparse ideal of $S$. Since $I$ is an $\fkn$-primary ideal and $S$ is a one-dimensional integral domain, we get $0<\depth_S(I)\le \dim_SI\le \dim S =1$. Hence, $I$ is a Cohen-Macaulay $S$-module of dimension one. On the other hand, since $g(I)+1=\ell_S(S/I)+2\ell_S(\overline{S}/S)$. It implies that $$g(\val(I))+1=|\val(S)\setminus \val(I)|+2\ell_S(\overline{S}/S).$$
Thus, $\val(I)$ is a maximum sparse ideal of the semigroup ring $\val(S)$. Therefore, $\val(I)=\val(S)\setminus D(i)$ for some $i$ such that $G(i)=0$. We get that $\val(I)$ is irreducible (by \cite[Proposition 1]{BaK10}). Hence, $I$ is irreducible as well (by \cite[Theorem 3]{BaK10}). We get $\ell_S(I:\fkn/I)=1$. 
Therefore, thanks to \cite[Satz 3.3]{HeK71}, $I$ is a canonical ideal of $S$.
\end{proof}

Now, let us give a characterization for the Gorensteinness of $S$. 
	
\begin{thm}\label{gor}
The following statements are equivalent.
\begin{enumerate}[$\rm 1)$]
\item $S$ is Gorenstein.
\item There is a maximum sparse ideal $I$ of $S$ such that $I$ is a principal ideal.
\item Every canonical ideal of $S$  is maximum sparse.
\end{enumerate}
\end{thm}
\begin{proof}
$1) \Rightarrow 2)$  Suppose that $S$ is Gorenstein. By Remark \ref{exists1} we take $I$ to be a maximum sparse ideal. Then, Lemma \ref{ex1} says that $I$ is a canonical ideal of $S$. Since the residue field $S/\n$ is infinite, there exists $Q=(a)$ a reduction of $I$. Then since $S$ is Gorenstein, we have  $I=Q$ (by \cite[Lemma 3.7]{GMP13}). Hence, $I$ is principal.\\
$2) \Rightarrow 1)$: Suppose $I=aS$ is a maximum sparse ideal of $S$. Then $I$ is irreducible. Hence, since $a$ is regular on $S$, we get $$r_S(S)=r_S(S/aS)=r_S(S/I)=\ell_S(I:\fkn/I)=1.$$
Therefore, $S$ is Gorenstein.\\
$1) \Rightarrow 3)$ Suppose that $S$ is Gorenstein and $I$ is a canonical ideal of $S$. Since $I\cong S$, we have $I=aS$ for some $a\in I$. Note that $a$ is unit in $Q(S)$, one has $$\ell_S(I/(I:\overline{S}))=\ell_S(aS/(aS:\overline{S}))=\ell_S(aS/a(S:\overline{S}))=\ell_S(S/\fkC).$$
Hence, on the one hand we have $\ell_S(S/\fkC) \le \ell_S(\overline{S}/S)$ (by Theorem \ref{main}). On the other hand, $\ell_S(\overline{S}/S)\le \ell_S(S/\fkC)$ (by \cite[Theorem 1]{BrH92}). It implies that $\ell_S(\overline{S}/S)= \ell_S(S/\fkC)$. Thus, $\ell_S(I/(I:\overline{S}))= \ell_S(\overline{S}/S)$ so that $I$ is maximum sparse.\\
$3) \Rightarrow 1)$ 
Taking the canonical module $\omega_S$ as in Lemma \ref{irr}. If $\omega_S=S$ then it is clear that $S$ is Gorenstein. If otherwise, we may let $\omega_S=S+\sum_{i=1}^{n}\frac{a_i}{b_i}S$ for some a positive integer $n$ and some $a_i,b_i\in S, b_i\neq 0$ for all $1\le i\le n $. Let $x=b_1b_2...b_n$. Then $x\omega_S$ is a canonical ideal of $S$. Let $xzS$ is a minimal reduction of $x\omega_S$, where $z\in\omega_S$. Then by (\cite[Lemma 1.11]{Del94}) we have $\ell_S(z\overline{S}/\omega_S)=\ell_S(S/\fkC)$. Moreover,
since $\fkC$ is an ideal of both $S$ and $\overline{S}$, we have  $$x\omega_S\fkC=x\omega_S\overline{S}\fkC=xz\overline{S}\fkC=xz\fkC.$$ Therefore,
\begin{align*}
\ell_S(\overline{S}/S)=&\ell_S(xz\overline{S}/xzS)\\
=&\ell_S(xz\overline{S}/x\omega_S)+\ell_S(x\omega_S/xzS)\\
=&\ell_S(S/\fkC)+\ell_S(x\omega_S/xzS)\\
=&\ell_S(x\omega_S/xz\fkC)\\
=&\ell_S(x\omega_S/x\omega_S\fkC)\\
\le&\ell_S(x\omega_S/\fkC_{x\omega_S}) (\text { because } x\omega_S\fkC\subseteq \fkC_{x\omega_S}).
\end{align*}
On the other hand, by the assumption, $x\omega_S$ is maximum sparse. Therefore we obtain the equality above. Hence,   $x\omega_S\fkC= \fkC_{x\omega_S}$. It follows that 
$$xz\fkC= {x\omega_S}:\overline{S}=x(\omega_S:\overline{S})=x\fkC.$$
We get $z\fkC=\fkC$. Note that $z$ is in $S$. If $z\in\m$ then by Nakayama's lemma we get $\fkC=0$ which is a contradiction. Therefore, $z$ is unit, whence $S=x\omega$. Consequently, $S$ is Gorenstein. 
\end{proof}
We consider Theorem \ref{gor} in the context of numerical semigroup rings. Let $H$ be a numerical semigroup. A relative ideal $I$ of $H$ is called a canonical ideal of $H$ if $I-(I-J)=J$ for every relative ideal $J$. A typical example of canonical ideals is the ideal $\Omega=\{g(S)-a\mid a\in \Z\setminus S\}$ which is called the standard canonical ideal of $S$. It is easy to see that if $S=k[[H]]$ be the numerical semigroup ring of $H$, then an ideal $I$ of $S$ is maximum sparse if and only if $\val(I)$ is maximum sparse of $H$. The ideal $I$ is a canonical ideal of $S$ if and only if $\val(I)$ is a canonical ideal of $H$ (cf. \cite{Jag77}). Moreover, $S$ is Gorenstein if and only if $H$ is symmetric (cf. \cite{Kun70}). Therefore, by applying Theorem \ref{gor} for the case $S=k[[H]]$, then we immediately get the following result in numerical semigroups. 
\begin{cor} Let $H$ be a numerical semigroup. Then the following assertions are equivalent.
\begin{enumerate}[$\rm 1)$]
	\item $H$ is symmetric.
	\item There is a maximum sparse ideal $I$ of $H$ such that $I$ is a principal ideal.
	\item Every canonical ideal of $H$  is maximum sparse.
\end{enumerate}
\end{cor}

Now we define sparse  stretched rings which is a special class of canonical  stretched rings.
\begin{defn}
We say that $(S,\n)$ is a {\it sparse stretched ring}, if there is a maximum sparse ideal $I\subseteq \n^2$ of $S$ such that   $S/I$ is a stretched ring.
\end{defn}

\begin{prop}
If $S$ is a  sparse  stretched ring then $S$ is  a canonical  stretched ring.
\end{prop}
\begin{proof}
It is now immediate from Lemma \ref{ex1} and the definition of sparse  stretched rings.
\end{proof}
Note that the reverse implication may not be true (see Example \ref{cannotssr}). 


\begin{thm}\label{main1}
Assume that the Hilbert function  of $S$ is non-decreasing. Then
the following  statements  are equivalent.
\begin{enumerate}[$\rm 1)$]
 \item $S$ is a sparse  stretched ring.
\item There exists an $\n$-primary ideal $I\subseteq \n^2$  such that   $\mu(\n^{v_S(I)}/I)=1$, $I:\n\subseteq \n^{v_S(I)}$ and
$$s(I)(\val(z)-1)=2\delta+ \edim(S)-1 $$
for some $z\in\n\setminus \n^{v_S(I)+1}$ and $z^2\not\in I$.
\end{enumerate}
When this is the case, $v_S(I)=2.$
\end{thm}
\begin{proof}
$1) \Rightarrow 2)$   Since $S$ is a sparse stretched ring, there exists  an irreducible ideal $I\subseteq \n^2$ of $S$ such that $S/I$ is a stretched ring. By Corollary \ref{stret}, we have $v_S(I)=2, \mu(\n^{v_S(I)}/I)=1$ and $I:\n\subseteq \n^{v_S(I)}$. 

Let $R=S/I$ and $\m=\n/I$. Then, by Lemma \ref{irr}, there exists a basis $z_1,\ldots, z_{\edim(S)}$ for $\n$ such that $\m^2=z_1^2R$. Thus, $\m^{s(R)}=z_1^{s(R)}R$. By the definition, we have 
$$g(I)=s(R)\val(z_1).$$
Since $S/I$ is a stretched ring and $I$ is a maximum sparse ideal, we have
$$\begin{aligned}
s(R)(\val(z_1)-1)&= g(I)-s(R)\\
&=2\delta +\ell_S(S/I)- s(R)-1\\
&=2\delta+ \edim(S)-1. 
\end{aligned}$$

$2) \Rightarrow 1)$ 
Let $R=S/I$ and $\m=\n/I$. It follows from $\mu(\m^{v_S(I)})=1$ and Corollary \ref{stret} that the Hilbert function of  $R$ is given by \ref{eq7}.  
Therefore, by Theorem 3.4 in \cite{Sha14} , we have
$$h_S(v_S(R)-1)-h_S(1)+1\le \dim_k(0:_R\fkm).$$
Moreover, since $0:_R\fkm \subseteq \m^{v_S(R)}$ and  $\mu(\m^{v_S(I)})=1$, we have $\dim_k(0:_R\m)=1$. Then,  $h_S(v_S(R)-1)-h_S(1)\le 0$. Since the Hilbert function $h_{S}$ of $S$ is non-decreasing, we have $v_S(R)=2$. Then, by Corollary \ref{stret},  $S/I$ is a stretched ring. 

Since $z\not\in \n^{v_S(I)+1}$, $z^2\not\in I$ and $\m^2$ is principal ideal, $\m^{s(R)}=z_1^{s(R)}R$. By the definition, we have 
$$
\begin{aligned}
g(I)&=s(R)\val(z)\\
&= 2\delta+ \edim(S)-1+s(R)\\
&= 2\delta +\ell_S(S/I)-1.
\end{aligned}$$
Hence, $I$ is  a maximum sparse ideal of $S$.
\end{proof}

\begin{ex}
We look at the local ring
$$S=k[[t^6,t^7,\ldots,t^{11}]]$$
in the formal power series ring $k[[t]]$ over a field $k$. We put $I=(t^{12},t^{13},t^{14},t^{15},t^{16})$. Then $g(I)=17$, $\ell_S(S/I)=8$ and $\delta=5$. Thus, we have
$$g(I)=\ell_S(S/I)+2\delta-1,$$ 
and $I$ is a maximum sparse ideal of $S$. Moreover $S/I$ is an Artinian stretched ring, because the Hilbert function of $S/I$ is 
$$
h_{S/I}(i)=\begin{cases}
1& \text{ if } i=0,\\
6& \text{ if } i=1,\\
1& \text{ if } i=2,\\
0& \text{ if } i\ge3.
\end{cases}$$
Hence, $S$ is a sparse  stretched ring.

\end{ex}

\begin{ex}\label{notSSR}
Let $S=k[[t^4,t^6+t^7,t^{15}]]\subseteq\overline{S}=k[[t]]$, where char$(k)\neq 2$. We have the value semigroup $\val(S)=\left<4,6,13,15\right>$.  By a direct computation, we get
$\ell_S(\overline{S}/S)=7$ and $\edim(S)=3$. Therefore 
$2\ell_S(\bar{S}/S)+\edim(S)-1=16.$ Assume that $S$ is a sparse stretched ring. Then, by Theorem \ref{main1}, there exist $I\subseteq \n^2$ and $z\in\n\setminus\n^{v_S(I)+1}$ such that $z^2\notin I$ and $s(I)(\val(z)-1)=16$. Hence, since $2,3,5,9\notin \val(S)$, we must have $\val(z)=17$ which implies $s(I)=1$. It follows that $I=\n^2$ so that $z^2\in I$ which is a contradiction. Therefore, $S$ is not a sparse  stretched ring. 
\end{ex}

\section{Canonical  stretched property of $3$-generated numerical semigroup rings}

 In this section, we analyze the semigroup ring $k[H]$ of a 3-generated numerical semigroup $H$ with $k$ is a field. Let $a_1,a_2,a_3\in \Bbb Z$ and assume that $0<a_1<a_2<a_3$ with GCD$(a_1,a_2,a_3)=1$. Let $H$ be the numerical semigroup generated by $a_1,a_2,a_3$, that is, $H=\langle a_1,a_2,a_3\rangle :=\{c_1a_1+c_2a_2+c_3a_3\mid 0\le c_i\in\Bbb Z, 1\le i\le 3\}$. Let $k[t]$ denotes the polynomial ring and put $T=k[t^{a_1},t^{a_2},t^{a_3}]\subseteq k[[t]]$. Then, $T$ is a one-dimensional graded domain with $\overline T=k[t]$, where $\overline T$ stands for the normalization of $T$. Let $M=(t^{a_1},t^{a_2},t^{a_3})$ denote the maximal ideal of $T$ generated by $t^{a_i}$'s. In this section, we explore the local ring $S=T_M$ and eventually answer the question of when $S=k[[t^{a_1},t^{a_2},t^{a_3}]]$ is 
a sparse/canonical stretched ring. Throughout, we assume that $S$ is not a Gorenstein ring. Let $U=k[[X,Y,Z]]$ be the polynomial ring and regard $U$ as a $\Bbb Z$-graded ring with $U_0=k$, $\deg X=a_1$, $\deg Y=a_2$, and $\deg Z=a_3$. Let $\varphi:U\to S$ be the $k$-algebra homomorphism defined by $\varphi(X)=t^{a_1}, \varphi(Y)=t^{a_2}, \text{ and }\varphi(Z)=t^{a_3}$. 
Hence, $\Im \varphi=S$.  We put $J=\Ker \varphi$. Then, because $S$ is not a Gorenstein ring, thanks to \cite{her70}, the ideal $J$ is generated by the maximal minors of the matrix (which is called the Herzog matrix)
$$\begin{pmatrix}
X^{\alpha}& Y^{\beta} & Z^{\gamma}\\
Y^{\beta^\prime}& Z^{\gamma^\prime} & X^{\alpha^\prime}
\end{pmatrix},$$
where $0<\alpha,\beta,\gamma,\alpha^\prime,\beta^\prime,\gamma^\prime\in\Bbb Z$. Let us call this matrix the Herzog matrix of $H$. Let $\Delta_1=Z^{\gamma+\gamma^\prime}-X^{\alpha^\prime}Y^\beta$, $\Delta_2=X^{\alpha+\alpha^\prime}-Y^{\beta^\prime}Z^{\gamma}$, and $\Delta_3=Y^{\beta+\beta^\prime}-X^\alpha Z^{\gamma^\prime}$. Then $J=(\Delta_1,\Delta_2,\Delta_3)$ and $S\cong U/J.$

We are in a position to summarize these arguments.
\begin{thm}\label{charcsr}
Let $S=k[[t^{a_1},t^{a_2},t^{a_3}]]$ is a numerical semigroup ring with embedding dimension three. Then $S$ is always a canonical stretched ring.
\end{thm}
\begin{proof}
Firstly, we see that the Herzog matrix of $S$ do not have the form
$$\begin{pmatrix}
	X& Y & Z\\
	Y& Z & X\\
\end{pmatrix}.$$
Indeed, suppose $S\cong U/\rmI_2(\begin{smallmatrix}
	X& Y & Z\\
	Y& Z & X\\
\end{smallmatrix}).$ Then one has (see \cite{her70}) 
$$\begin{cases}
2a_1=a_2+a_3\\
2a_2=a_1+a_3\\
2a_3=a_1+a_2
\end{cases}.$$ Hence, $a_1=a_2=a_3=1$ (by gcd$(a_1,a_2,a_3)=1$). It implies that $S=k[[t]]$. Therefore, $k[[t]]\cong k[[X,Y,Z]]/(X^2-YZ,Y^2-XZ,Z^2-XY)$. It is clear that the right hand side is not a DVR (because it has embedding dimension 3) which is a contradiction to the left hand side.

Hence, $\max\{\alpha+\alpha^\prime,\beta+\beta^\prime,\gamma+\gamma^\prime\}\ge 3$.
Then, we may assume that $\alpha+\alpha^\prime\ge 3$.
Let $I=(XY,YZ,ZX,X^s+Y^2,X^s+Z^2)$ for $2\le s\le\alpha+\alpha^\prime-1$. Then 
$X^{s+1}=X(X^s+Y^2)-XY^2\in I$, $Y^3=Y(X^s+Y^2)-X^sY\in I$ and $Z^3=Z(X^s+Z^2)-X^sZ\in I$. Thus $J\subseteq I$. 

Let $\n=(X,Y,Z)U$ and $\m=\n/J$.  Then $I\subseteq \n^2$, $I:\n=(XY,YZ,ZX,X^s,Y^2,Z^2)$ and so $\ell_U((I:\n)/I)=1$. Hence, for all $2\le s\le \alpha+\alpha^\prime-1$, $I$ is a canonical ideal of $U$. Moreover, the Hilbert function of  $U/I$ is
\begin{equation*}\label{eq14}
\begin{tabular}{|c|c|c|c|c|c|c|c|c|c|}
\hline
$0$&$1$&$2$&$\ldots$&$s$&$s+1$\\
\hline
$1$&$3$&$1$&$\ldots$&$1$&$0$\\
\hline
\end{tabular}
\end{equation*}
Hence, $S$ is a canonical  stretched ring.

\end{proof}
We close this paper with the following two examples of canonical stretched rings. One of them is not a sparse stretched ring and the other is not a stretched ring.
\begin{ex}\label{cannotssr}
Let $S=k[[t^4,t^5,t^7]]$ is a numerical semigroup ring over a field $k$ with the maximal ideal $\n=(t^4,t^5,t^{7})$. 
Thanks to Theorem \ref{charcsr}, $S$ is a canonical stretched ring. But  $S$ is not a sparses stretched ring. Indeed, assume that $S$ is a sparse stretched ring. Then, by Theorem \ref{main1}, there exist $I\subseteq \n^2$ and $z\in\n\setminus\n^{v_S(I)+1}$ such that $z^2\notin I$ and $s(I)(\val(z)-1)=2\delta+\edim(S)-1$. Since $\delta = 4,\edim S=3$, $s(I)(\val(z)-1)=2\delta+\edim(S)-1=10$.  Thus since $2,3,6\notin \val(S)$, we must have $\val(z)=11$ and $s(I)=1$. It follows that $I=\n^2$ and so $z^2\in I$, which is a contradiction. Hence $S$ is not a sparses stretched ring.
\end{ex}
\begin{ex}\label{cannotstr}
	Let $S=k[[t^5,t^6,t^{13}]]$ is a numerical semigroup ring over a field $k$ with the maximal ideal $\n=(t^5,t^6,t^{13})$. Then by Theorem \ref{charcsr}, $S$ is a canonical stretched ring. But $S$ is not a stretched ring. In fact, we can check that $\n$ contains a unique minimal reduction $J=(t^5)$. Notice that  $\ell_S(\frac{\n^2+J}{\n^3+J}) = 2$. 
\end{ex}

\section*{Acknowledgments} 

 The second author was partially supported by  the Vietnam National Foundation for Science and Technology Development (NAFOSTED) under grant number 101.04-2019.309. The third author was partially supported by the Alexander von Humboldt Foundation and the Vietnam National Foundation for Science and Technology Development (NAFOSTED) under grant number 101.04-2019.309.
The authors would like to thank the referee for the valuable comments to improve this
article.
The Theorem \ref{charcsr} is largely due to the inspiring suggestions
of Naoyuki Matsuoka.


\end{document}